\RequirePackage{fix-cm}
\documentclass[smallextended]{svjour3}       % onecolumn (second format)
\smartqed  % flush right qed marks, e.g. at end of proof
\usepackage{graphicx}
\usepackage{makeidx}
\usepackage{graphics}
\graphicspath{ {images/} }
\usepackage{amsmath}
\usepackage{amsfonts}
\usepackage{amssymb}
\usepackage{epsfig}
\usepackage{tabularx}
\usepackage{latexsym}
\usepackage{mathtools}
\usepackage{epstopdf}
\usepackage{array}
\usepackage[utf8]{inputenc}
\usepackage{caption}
\usepackage{enumerate}
\usepackage[english]{babel}
\usepackage[autostyle]{csquotes}
\usepackage{float}
\usepackage{pgf, tikz}
\usetikzlibrary{arrows, automata}
\newtheorem{observation}{Observation}
\newcommand{\floor}[1]{\left\lfloor #1\right \rfloor}
\newcommand{\ceil}[1]{\left\lceil #1 \right\rceil}

\begin{document}

\title{Rectilinear Crossings in Complete Balanced $d$-Partite $d$-Uniform Hypergraphs%\thanks{Grants or other notes
}


\author{Rahul Gangopadhyay      \and
        Saswata Shannigrahi%etc.
}


\institute{Rahul Gangopadhyay \at
              IIIT Delhi, India \\
              \email{rahulg@iiitd.ac.in}           %  \\
           \and
           Saswata Shannigrahi \at
           Saint Petersburg State University, Russia\\
            \email{saswata.shannigrahi@gmail.com}  
}

\date{Received: date / Accepted: date}

\maketitle

\begin{abstract}
 In this paper, we study the embedding of a complete balanced $d$-partite $d$-uniform hypergraph with  its $nd$ vertices represented as points in general position in $\mathbb{R}^d$ and each hyperedge drawn as the convex hull of $d$ corresponding vertices. We assume that the set of vertices is partitioned into $d$ disjoint sets, each of size $n$, such that each vertex in a hyperedge is from a different set. Two hyperedges are said to be crossing if they are vertex disjoint and contain a common point in their relative interiors.  Using  Colored Tverberg theorem with restricted dimensions, we observe that such an embedding of a complete balanced $d$-partite $d$-uniform hypergraph with $nd$ vertices contains $\Omega\left((8/3)^{d/2}\right){\left({n/2}\right)^d{\left((n-1)/2\right)}^d}$ crossing pairs of hyperedges for $n \geq 3$ and sufficiently large $d$. Using  Gale transform and  Ham-Sandwich theorem, we improve this lower bound to $ \Omega\left(2^{d}\right){\left({n/2}\right)^d{\left((n-1)/2\right)}^d}$ for $n \geq 3$ and sufficiently large $d$. 
\end{abstract}

\section{Introduction}
\label{intro}
The rectilinear drawing of  a graph is defined as an embedding of it in $\mathbb{R}^2$ such that its vertices are represented as points in general position (i.e., no three vertices are collinear) and edges are drawn as straight line segments connecting the corresponding vertices. The rectilinear crossing number of a graph $G$, denoted by $\overline {cr}(G)$, is defined as the minimum number of  crossing pairs of edges among all rectilinear drawings of $G$. Determining the rectilinear crossing number of a graph is one of the most important problems in graph theory \cite{AC,JN,KL,Naha}. In particular, finding the rectilinear crossing numbers of  complete bipartite graphs is an active area of research \cite{KL,RT}. Let $K_{n,n}$ denote the complete bipartite graph having  $n$ vertices in each part.  For any $n \geq 5$, the best-known lower and upper bounds on $\overline {cr}(K_{n,n})$ are $\left(n(n-1)/{5}\right)\floor{n/2}\floor{(n-1)/2}$ and $\floor{n/2}^2\floor{(n-1)/2}^2$, respectively \cite {KL,Zaran}. For sufficiently large $n$, the result of  Nahas \cite{Naha} improved the lower bound on $\overline {cr}(K_{n,n})$ to $\left(n(n-1)/{5}\right)$  $\floor{n/2}\floor{(n-1)/2}+ 9.9 \times 10^{-6}n^4$. 
\par  Hypergraphs are natural generalizations of graphs. A hypergraph $H$ is a pair $(V,E)$, where $V$ is a set of vertices and $E$ is a set of  distinct subsets of $V$ called hyperedges. A hypergraph $H$ is called \textit{$d$-uniform} if each hyperedge contains $d$ vertices. A $d$-uniform hypergraph $H=(V,E)$ is said to be \textit{$d$-partite} if there exist sets $X_1,X_2, \ldots, X_d$ such that $V= \bigcup\limits_{i=1}^{d} X_{i}$, $X_i \cap X_j= \emptyset$ for any $i \neq j$, and each vertex in a hyperedge belonging to $E$ is from a different $X_i$. We call $X_i$ to be the \textit{$i^{th}$ part} of $V$. Moreover, such a $d$-partite $d$-uniform hypergraph is called  \textit{balanced} if $\left|X_1\right|=\left|X_2\right|=\ldots=\left|X_d\right|$ and  \textit{complete} if $|E| = \left|X_1 \times X_2 \times \ldots \times X_d\right|$. The complete balanced $d$-partite $d$-uniform hypergraph with $n$ vertices in each part is denoted by $K_{d\times n}^d$. For $t \geq 2$, let us denote by $K_{k_1 \times n_1+k_2 \times n_2+ \ldots+ k_t \times n_t}^d$ the complete $d$-partite $d$-uniform hypergraph if $\sum\limits_{i=1}^t k_i = d$, $n_i \neq n_{i+1}$ for all $i$ in the range $1 \leq i \leq t-1$, and each of the first $k_1 > 0$ parts contains $n_1$ vertices, each of the next $k_2 > 0$ parts contains $n_2$ vertices, $\ldots$, each of the final $k_t > 0$ parts contains $n_t$ vertices.  \par
A \textit{$d$-dimensional rectilinear drawing}  \cite{SA} of a $d$-uniform hypergraph $H=(V,E)$ is an embedding of it in $\mathbb{R}^d$ such that its vertices are represented as points  in general position (i.e., no $d+1$ vertices lie on a hyperplane) and its hyperedges are drawn as $(d-1)$-simplices formed by the corresponding vertices. In a $d$-dimensional rectilinear drawing of $H$, two hyperedges  are said to be \textit {crossing} if they are vertex disjoint and contain a common point in their relative interiors \cite{DP}. The \textit{$d$-dimensional rectilinear crossing number} of  $H$, denoted by ${\overline {cr}_d}(H)$, is defined as the minimum number of crossing pairs of hyperedges among all $d$-dimensional rectilinear drawings of $H$. Let us mention a few existing results on the $d$-dimensional rectilinear crossing number of uniform hypergraphs. Dey and Pach \cite{DP}  proved that ${\overline {cr}_d}(H)=0$ implies the total number of hyperedges in $H$ to be $O(\left|V\right|^{d-1})$. Gangopadhyay et al. \cite{GS} showed that  ${\overline {cr}_d}(H)=\Omega\left(2^d \sqrt{d}\right)\dbinom{|V|}{2d}$ when $H$ is a complete $d$-uniform hypergraph. \par
 In this paper, we first use  Colored Tverberg theorem with restricted dimensions and Lemma \ref{HDC} to observe a lower bound on $\overline{cr}_d\left(K_{d \times n}^d\right)$ for $n \geq 3$ and sufficiently large $d$. We mention this lower bound in Observation \ref{gencased}. Let us introduce a few more definitions and notations used in its proof. Two $d$-uniform hypergraphs $H_1=(V_1,E_1)$ and $H_2=(V_2, E_2)$ are \textit{isomorphic} if there is a bijection $f: V_1 \rightarrow V_2$ such that any set of $d$ vertices $\{u_1, u_2, \ldots, u_d\}$ is a hyperedge in $E_1$ if and only if $\{f(u_1), f(u_2), \ldots, f(u_d)\}$ is a hyperedge in $E_2$. A hypergraph $\tilde{H}=(\tilde{V},\tilde{E})$ is called an  \textit{induced sub-hypergraph} of $H=(V,E)$ if $\tilde{V}\subseteq V$ and $\tilde{E}$ is the set of all hyperedges in $E$ that are formed only by the vertices in $\tilde{V}$. The convex hull of a finite point set $S$  is denoted by $Conv(S)$. The convex hulls $Conv(S)$ and $Conv(S')$ of two finite point sets $S$ and $S'$ {\textit {intersect}} if they contain a common point in their relative interiors. For $u$ and $w$ in the range $2 \leq u, w \leq d$,  a $(u-1)$-simplex $Conv(U)$ spanned by a point set $U$ containing $u$ points and a $(w-1)$-simplex $Conv(W)$ spanned by a point set $W$ containing $w$ points (when these $u+w$ points are in general position in $\mathbb{R}^d$)  \textit{cross} if $Conv(U)$ and $Conv(W)$ intersect, and $U \cap W= \emptyset$ \cite{DP}.  
 \vspace{0.2cm}\\
\noindent\textbf {Colored Tverberg Theorem with restricted dimensions.}~\cite{JM1,VZ}
Let $\{C_1, C_2, \ldots, C_{k+1}\}$ be a collection of $k+1$ disjoint finite point sets in $\mathbb{R}^d$. Each of these sets is  of cardinality at least $2r-1$, where $r$ is a prime number satisfying the inequality $r(d-k)\leq d$. Then, there exist $r$ disjoint sets $S_1, S_2, \ldots, S_r$ such that $S_i\subseteq \bigcup_{j=1}^{k+1}C_j$, $\bigcap_{i=1}^{r} Conv(S_i) \neq \emptyset$ and $\left|S_i \cap C_j \right|=1$ for all $i$ and $j$ satisfying $ 1\leq i \leq r$ and $1 \leq j \leq k+1$.

\begin{lemma}\cite{AGSV}
\label{HDC}
Consider two disjoint point sets $U$ and $W$, each a subset of a set  $A$ containing 2d points in general position in $\mathbb{R}^d$, such that $|U|= u$, $|W|= 
w$, $2 \leq u, w \leq d$ and $u+w \geq d+1$. If the $(u-1)$-simplex formed by 
$U$ crosses the 
$(w-1)$-simplex formed by $W$, then the $(d-1)$-simplices formed by any 
two disjoint point sets $U'\supseteq U$ and $W' \supseteq W$ satisfying $|U'| = 
|W'| =d$ and $U', W' \subset A$  also cross.
\end{lemma}
\begin{observation}\label{gencased}
 $\overline{cr}_d\left(K_{d \times n}^d\right) = \Omega\left((8/3)^{d/2}\right){\left({n/2}\right)^d{\left((n-1)/2\right)}^d}$ for $n \geq 3$ and sufficiently large $d$.
\end{observation}
 
 \begin{proof}
Let us consider the hypergraph $H=K_{d \times n}^d$ such that its vertices are in general position in $\mathbb{R}^d$. Let $H'=K_{(\ceil{d/2}+1) \times 3+(\floor{d/2}-1) \times 2}^d$ be an induced sub-hypergraph of it containing $3$ vertices from each of the first $\ceil{d/2}+1$ parts and 2 vertices  from each of the remaining $\floor{d/2}-1$ parts. Let  $C_i$ denote the $i^{th}$ part of the vertex set of $H'$ for each $i$ in the range $1 \leq i \leq \ceil{d/2}+1$. Note that  $C_1, C_2, \ldots, C_{\ceil{d/2}+1}$ are disjoint sets in $\mathbb{R}^d$ and each of them contains $3$ vertices. Clearly, these sets satisfy the condition of  Colored Tverberg theorem with restricted dimensions for $k=\ceil{d/2}$ and $r=2$. Since the vertices of $H'$ are in general position in $\mathbb{R}^d$,  Colored Tverberg theorem with restricted dimensions implies that there exists a  crossing pair of  $\ceil{d/2}$-simplices spanned by $U\subseteq \bigcup_{j=1}^{\ceil{d/2}+1}C_j$ and $W\subseteq \bigcup_{j=1}^{\ceil{d/2}+1}C_j$ such that $U \cap W = \emptyset$ and $\left|U \cap C_j \right| =1$, $\left|W \cap C_j\right|=1$ for each $j$ in the range $ 1\leq j \leq  \ceil{d/2}+1$. Lemma \ref{HDC} implies that $U$ and $W$ can be extended to form $2^{{\floor{d/2}}-1}$ distinct crossing pairs of $(d-1)$-simplices, where each $(d-1)$-simplex contains exactly one vertex from each part of  $H'$.  This implies that  $\overline{cr}_d\left(H'\right) \geq 2^{{\floor{d/2}}-1}$.
Note that each crossing pair of hyperedges corresponding to these $(d-1)$-simplices is contained in $(n-2)^{\ceil{d/2}+1}$ distinct induced sub-hypergraphs of $H$, each of which is isomorphic to $H'$. Moreover, there are ${\dbinom{n}{3}}^{\ceil{d/2}+1}{\dbinom{n}{2}}^{\floor{d/2}-1}$ distinct induced sub-hypergraphs of $H$, each of which is isomorphic to $H'$.  This implies $\overline{cr}_d\left(K_{d \times n}^d\right) \geq {\scriptsize 2^{\floor{d/2}-1} {\dbinom{n}{3}}^{\ceil{d/2}+1} {\dbinom{n}{2}}^{\floor{d/2}-1}\bigg /(n-2)^{\ceil{d/2}+1}}$ $={n^d(n-1)^d}/{6^{{\ceil{d/2}}+1}}= \Omega \left((8/3)^{d/2}\right)$${\left({n/2}\right)^d{\left((n-1)/2\right)}^d}$. \qed
\end{proof}
In Section \ref{gencase}, we improve the lower bound on $\overline{cr}_d\left(K_{d \times n}^d\right)$ as follows. To the best of our knowledge, this is the first non-trivial lower bound on this number.
\begin{theorem}
\label{improved}
 $ \overline{cr}_d\left(K_{d \times n}^d\right) = \Omega\left(2^{d}\right){\left({n/2}\right)^d{\left((n-1)/2\right)}^d}$ for $n \geq 3$ and sufficiently large $d$.
 \end{theorem}

\section{Techniques Used}
\label{sec:1}
Let $A=$ $< a_1,a_2, \ldots, a_m>$ be a sequence of $m \geq d+1$ points in $\mathbb{R}^d$ such that the points in $A$ affinely span $\mathbb{R}^d$. Let the coordinate of the $i^{th}$ point $a_i$ be $(x_1^i,x_2^i, \ldots, x_d^i)$. To compute a \textit{ Gale transform} \cite{GL} $D(A)$ of $A$, we consider the following matrix $M(A)$.

\begin{center}
$
M(A) = 
\begin{bmatrix}
x_1^1&x_1^2&\cdots & x_1^m  \\
x_2^1&x_2^2&\cdots & x_2^m \\
\vdots & \vdots & \vdots & \vdots\\ 
x_d^1&x_d^2&\cdots & x_d^m \\
1 & 1 & \cdots & 1
\end{bmatrix}
$
\end{center}

Note that the rank of $M(A)$ is $d+1$, since the points in $A$ affinely span $\mathbb{R}^d$. The Rank-nullity theorem \cite{Mey}
implies that the dimension of the null space of $M(A)$ 
is $m-d-1$.  Let $\{(b_1^1, b_2^1, \ldots, b_m^1),$  $(b_1^2, b_2^2, \ldots, 
b_m^2), \ldots,(b_1^{m-d-1}, b_2^{m-d-1},$ $\ldots, b_m^{m-d-1}) \}$ be a 
basis of this null space. The Gale transform $D(A)$ corresponding to this basis is the sequence of $m$ vectors $\big<(b_1^1$, $b_1^2$, $\ldots, 
b_1^{m-d-1}),$ $(b_2^1,$ $b_2^2,$ $\ldots,$ $b_2^{m-d-1}),$ $\ldots,$ $(b_m^1, 
b_m^2, 
\ldots,$ $b_m^{m-d-1})\big>$. Note that $D(A)$ can also be considered as a point sequence in 
$\mathbb{R}^{m-d-1}$, which we denote by $<g_1,g_2,, \ldots, g_m>$. 

\noindent In the following, we mention some properties of $D(A)$ that are used in Section \ref{gencase}. For the sake of completeness, we give a proof of Lemma \ref{Radonpart}. On the other hand, the proof of Lemma \ref{genposi} is straightforward and we omit it here.

\begin{lemma}\cite{GRU}
\label{genposi}
If the points 
in $A$ are in general position in $\mathbb{R}^d$, each collection of $m-d-1$ vectors in $D(A)$ spans $\mathbb{R}^{m-d-1}$. 
\end{lemma}

\begin{lemma}\cite{GRU}
\label{Radonpart}
Let $h$ be a linear hyperplane, i.e., a hyperplane passing through the origin, in $\mathbb{R}^{m-d-1}$.  Let  $D^+(A) \subset D(A)$ and $D^-(A)\subset D(A)$ denote two sets of  vectors such that $\left|D^+(A)\right|, \left|D^-(A)\right| \geq 2$ and the vectors in $D^+(A)$ and $D^-(A)$ lie in the opposite open half-spaces $h^+$ and $h^-$ created by $h$, respectively. Then, the convex hull of the point set $A_a=\{a_i | a_i \in A, g_i \in D^+(A)\}$ and the convex hull of the point set $A_b=\{a_j | a_j \in A, g_j \in D^-(A)\}$ intersect.  
\end{lemma}

\begin{proof}
Let us assume that the hyperplane 
$h$ is given by the equation $\sum\limits_{i =1}^{m-d-1} \alpha_ix_i = 0$ such that $ \alpha_i \neq 0$ for at least one $i$, and $h^+$ ($h^-$) is the positive (negative) open half-space created by it. Let $D^0(A)= \{g_k | g_k \in D(A), g_k~lies~on~h\}$. This implies that there exists a vector $(\mu_1, 
\mu_2, \ldots, \mu_m)=\alpha_1(b_1^1,$  $b_2^1, \ldots, b_m^1) + 
\alpha_2(b_1^2,$ $b_2^2,$ $\ldots,$ $b_m^2)$ $+ \ldots +$ $
\alpha_{m-d-1}(b_1^{m-d-1},$ 
$b_2^{m-d-1},$ $\ldots,$ $b_m^{m-d-1})$ such that $\mu_i>0$ for each $g_i \in D^+(A)$,  $\mu_j < 0$ for each $ g_j \in D^-(A)$ and $\mu_k=0$ for each $g_k \in D^0(A)$. Since this vector $(\mu_1, 
\mu_2, \ldots, \mu_m)$ lies in the null space  of 
$M(A)$, it satisfies the following equation.
\vspace{-0.5cm}
\begin{center}
\begin{equation*}\label{eqn:glprop1}
 \begin{bmatrix}
x_1^1&x_1^2&\cdots & x_1^m  \\
x_2^1&x_2^2&\cdots & x_2^m \\
\vdots & \vdots & \vdots & \vdots\\ 
x_d^1&x_d^2&\cdots & x_d^m \\
1 & 1 & \cdots & 1
 \end{bmatrix}
     \begin{bmatrix}
      \mu_1 \\
       \mu_2 \\
        \vdots\\
       \mu_m\\
  \end{bmatrix}
  =
  \begin{bmatrix}
       0 \\
      0 \\
      \vdots\\[2ex]
      0
  \end{bmatrix}
\end{equation*}
\end{center}
From the equation above, we obtain the following.
\begin{equation*}
\label{eq2}
\sum\limits_{i:g_i \in D^+(A)} \mu_ia_i  = 
\sum\limits_{j:g_j \in D^-(A)} -\mu_ja_j, 
\sum\limits_{i:g_i \in D^+(A)} \mu_i  = \sum\limits_{j:g_j \in D^-(A)} -\mu_j 
\end{equation*}
Rearranging the equations above, we obtain the following.
\begin{equation*}
\sum\limits_{i:g_i \in D^+(A)} {\dfrac{\mu_i}{\sum\limits_{{i:g_i \in D^+(A)}} \mu_i}}a_i  = 
\sum\limits_{j:g_j \in D^-(A)} {\dfrac{\mu_j}{\sum\limits_{j:g_j \in D^-(A)} \mu_j}}a_j 
\end{equation*}

\begin{equation*}
\sum\limits_{i:g_i \in D^+(A)} {\dfrac{\mu_i}{\sum\limits_{i:g_i \in D^+(A)} \mu_i}}  = \sum\limits_{j:g_j \in D^-(A)} {\dfrac{\mu_j}{\sum\limits_{j:g_j \in D^-(A)} \mu_j}} = 1
\end{equation*}

\noindent It shows that $Conv(A_a)$ and $Conv(A_b)$ intersect.  \qed
\iffalse
Let $P_0=\{p_i| p_i \in P, i \in I^0 \}$. Let $P_0^+ \subseteq P_0 $ and $P_0^- \subseteq P_0$, such that $P_0^+ \cap P_0^-= \emptyset$ and $P_0^+ \cup P_0^-=P_0$. Let $P_a^+=P_a \cup P_0^+$ and     
$P_b^-=P_b \cup P_0^-$. Note that there exists a crossing between the convex hull of $P_a^+$ and the convex hull of $P_b^-$ and $P_a^+ \cap P_b^-= \emptyset$. Similarly, we define $P_a^-=P_a \cup P_0^-$ and     
$P_b^+=P_b \cup P_0^+$. Also note that there exists a crossing between the convex hull of $P_a^-$ and the convex hull of $P_b^+$ and $P_a^- \cap P_b^+= \emptyset$.  
\fi     
\end{proof}
\noindent Let us also state  Ham-Sandwich theorem here.
\par\vspace{0.4 cm}
\noindent\textbf {Ham-Sandwich Theorem.}~\cite{JM1,STO}
\textit{There exists a $(d-1)$-dimensional hyperplane $h$ which simultaneously bisects $d$ 
finite point sets $P_1, P_2, \ldots, P_d$ in $\mathbb{R}^d$, such that each of 
the 
open half-spaces created by $h$ contains at most 
$\left\lfloor\small{{|P_i|}/{2}}\right\rfloor$ 
points of $P_i$ for each $i$ in the range $1 \leq i \leq d$.}
\section {Lower Bound on the $d$-Dimensional Rectilinear Crossing Number of ${K_{d \times n}^d}$ }
\label{gencase}
\vspace{-0.2 cm}
In this section, we use  Gale transform of a sequence of points and  Ham-Sandwich theorem to improve the previously observed lower bound on the $d$-dimensional rectilinear crossing number of $K_{d \times n}^d$ for $n \geq 3$ and sufficiently large $d$. 
\vspace{0.4 cm}\\
{\bf{Proof of Theorem \ref{improved}.}}
Let us consider the hypergraph $H=K_{d \times n}^d$ such that its vertices are in general position in $\mathbb{R}^d$. Let $H'=K_{2 \times 3+ (d-2) \times 2}^d$ be an induced sub-hypergraph of it containing $3$ vertices from each of the first $2$ parts and 2 vertices  from each of the remaining $(d-2)$ parts of the vertex set of $H$. Let $V'=$ $<v_1,v_2,v_3, \ldots, v_{2d+1},v_{2d+2}>$ be a sequence of the vertices of  $H'$ such that $\{v_1,v_2,v_3\}$ belongs to the first part $L_1$, $\{v_4,v_5,v_6\}$ belongs to the second part $L_2$ and $\{v_{2k+1},v_{2k+2}\}$ belongs to the $k^{th}$ part $L_k$ for each $k$ in the range $3 \leq k \leq d$. We consider a Gale transform of $V'$ and obtain a sequence of $2d+2$ vectors $D(V')=$ $<p_1,p_2,p_3, \ldots, p_{2d+1},p_{2d+2}>$ in $\mathbb{R}^{d+1}$. It follows from Lemma \ref{genposi} that any set containing $d+1$ of these vectors spans $\mathbb{R}^{d+1}$. As mentioned before, $D(V')$ can also be considered as a sequence of points in $\mathbb{R}^{d+1}$. In order to apply  Ham-Sandwich theorem in $\mathbb{R}^{d+1}$, we color the origin with color $c_0$,  $\{p_1,p_2,p_3\}$ with  color $c_1$, $\{p_4,p_5,p_6\}$ with color $c_2$ and $\{p_{2k+1},p_{2k+2}\}$ with color $c_k$ for each $k$ in the range $3 \leq k \leq d$.
 It follows from Ham-Sandwich theorem that there exists a hyperplane $h$ such that it passes through the origin and bisects the set  colored with $c_i$ for each $i$ in the range $1\leq i \leq d$. Note that at most $d$ points of $D(V')$ lie on the linear hyperplane $h$, since any set of $d+1$ vectors in $D(V')$ spans $\mathbb{R}^{d+1}$. This implies that there exist at least $d+2$ points of $D(V')$ that lie in $h^+ \cup h^-$, where $h^+$ is the positive open half-space and $h^-$ is the negative open half-space created by $h$. Let  $D^+(V')$ and $D^-(V')$ be the two sets of points lying in $h^+$ and $h^-$, respectively. It follows from Ham-Sandwich theorem that at most $d$ points of $D(V')$ can lie in either of $h^+$ and $h^-$. This implies that $\left|D^+(V')\right| \geq 2$ and $\left|D^-(V')\right| \geq 2$. Moreover, note that $2$ points having the same color cannot lie in the same open half-space. Lemma \ref{Radonpart} implies that there exist a $(u-1)$-simplex $Conv(V'_a)$ spanned by some vertices  $V'_a \subset V'$ and a  $(w-1)$-simplex $Conv(V'_b)$ spanned by some vertices  $V'_b \subset V'$ such that the following conditions are satisfied.
\begin{enumerate}[(I)]
  \item $V'_a \cap V'_b = \emptyset$
  \item $Conv(V'_a)$ and $Conv(V'_b)$ cross. 
  \item $2 \leq |V'_a|, |V'_b| \leq d$, $|V'_a|+|V'_b| \geq d+2$
  \item $|V'_a \cap L_i | \leq 1$ for each $i$ in the range $1 \leq i \leq d$
  \item $|V'_b \cap L_i | \leq 1$ for each $i$ in the range $1 \leq i \leq d$
\end{enumerate}

Lemma \ref{HDC} implies that the crossing between $Conv(V'_a)$ and $Conv(V'_b)$ can be extended to a crossing pair of $(d-1)$-simplices spanned by  any two disjoint vertex sets $U', W' \subset V'$ satisfying $\left|U'\right|=\left|W'\right|=d$ and $U' \supseteq V'_a$ and $W' \supseteq V'_b$, respectively. In fact, it is always possible to add vertices to $V'_a$ and $V'_b$ in such a way that the following conditions hold for $U'$ and $W'$.  
\begin{enumerate}[(I)]
\item $U' \cap W' = \emptyset$
\item $Conv(U')$ and $Conv(W')$ cross. 
\item $\left|U'\right|=\left|W'\right|=d$
\item $|U' \cap L_i | =1$ for each $i$ in the range $1 \leq i \leq d$
 \item$|W' \cap L_i | = 1$ for each $i$ in the range $1 \leq i \leq d$
\end{enumerate}
\indent ~~The argument above establishes the fact that $\overline{cr}_d(H') \geq 1$. Note that  $H$ contains ${\dbinom{n}{3}}^{2}{\dbinom{n}{2}}^{d-2}$ distinct induced sub-hypergraphs, each of which is isomorphic to $H'$. Since each crossing pair of hyperedges  is contained in $(n-2)^2$ distinct induced sub-hypergraphs of $H$, each of which is isomorphic to $H'$, we obtain $\overline{cr}_d\left(K_{d \times n}^d\right) \geq {\dbinom{n}{3}}^{2}{\dbinom{n}{2}}^{d-2}\bigg /(n-2)^{2} =(1/9){n^d(n-1)^d}/{2^d}= \Omega\left(2^{d}\right){\left({n/2}\right)^d{\left((n-1)/2\right)}^d}$. \qed

\section*{Acknowledgement}
The authors would like to thank the anonymous reviewer for valuable comments and suggestions, especially for suggesting the notations used for complete $d$-partite $d$-uniform hypergraphs in this paper.

%\begin{acknowledgements}
%\end{acknowledgements}



\iffalse
\begin{figure}[H]
      \centering
      \includegraphics[scale=0.4]{conv.pdf}
      \caption{A $3$-dimensional rectilinear  drawing of $K_{1 \times 3+ 2\times 2}^3$ }
      \label{fig6}
    \end{figure}  
\fi

\end{document}